\documentclass[10pt]{amsart}

\usepackage[utf8]{inputenc}

\usepackage{amsmath,amssymb,amsthm}
\usepackage{enumitem}
\usepackage{graphicx}

\newcommand{\inte}[2][X]{\mathrm{int}_{#1}\!\left(#2\right)}
\newcommand{\Ec}{\mathfrak{E}_c}

\newtheoremstyle{theorem}
     {11pt}
     {11pt}
     {}
     {}
     {\bfseries}
     {}
     {.5em}
     {\noindent\thmnumber{#2}. \thmname{#1}{\rm\thmnote{#3}}}

\theoremstyle{theorem}

\newtheorem{thm}{Theorem}
\newtheorem{defi}[thm]{Definition}
\newtheorem{obs}[thm]{Observation}
\newtheorem{coro}[thm]{Corollary}
\newtheorem{lemma}[thm]{Lemma}
\newtheorem{prop}[thm]{Proposition}
\newtheorem{claim}{Claim}[thm]

\newtheorem{ques}[thm]{Question}

\begin{document}

\title{Smooth fans that are endpoint rigid}
\author[Hern\'andez-Guti\'errez]{Rodrigo Hern\'andez-Guti\'errez}
\address[Hern\'andez-Guti\'errez]{Departamento de Matemáticas, Universidad Autónoma Metropolitana campus Iztapalapa, Av. San Rafael Atlixco 186, Leyes de Reforma 1a Sección, Iztapalapa, 09310, Mexico city, Mexico}
\email[Hern\'andez-Guti\'errez]{rod@xanum.uam.mx}
\author[Hoehn]{Logan C. Hoehn}
\address[Hoehn]{Nipissing University, Department of Computer Science \& Mathematics, 100 College Drive, Box 5002, North Bay, Ontario, Canada, P1B 8L7}
\email[Hoehn]{loganh@nipissingu.ca}

\hyphenation{Me-tro-po-li-ta-na}

\makeatletter
\@namedef{subjclassname@2020}{\textup{2020} Mathematics Subject Classification}
\makeatother


\begin{abstract}
Let $X$ be a smooth fan and denote its set of endpoints by $E(X)$.  Let $E$ be one of the following spaces: the natural numbers, the irrational numbers, or the product of the Cantor set with the natural numbers.  We prove that there is a smooth fan $X$ such that $E(X)$ is homeomorphic to $E$ and for every homeomorphism $h \colon X \to X$, the restriction of $h$ to $E(X)$ is the identity.  On the other hand, we also prove that if $X$ is any smooth fan such that $E(X)$ is homeomorphic to complete Erd\H{o}s space, then $X$ is necessarily homeomorphic to the Lelek fan; this adds to a 1989 result by W\l{}odzimierz Charatonik.
\end{abstract}

\date{\today}
\dedicatory{Dedicated to Professor W\l{}odzimierz Charatonik (1957--2021).}

\maketitle

\section{Introduction}

Recall that a space $X$ is \emph{homogeneous} if for every two points $x,y\in X$ there exists a homeomorphism $h\colon X\to X$ with $h(x)=y$. The study of homogeneous spaces in topology has a long history and it is still a relevant topic nowadays. For example, see the survey on homogeneity \cite{arh-vm-homogeneity}.

Many of the classical examples of spaces are homogeneous; for example: the real line, the circle, the Cantor set, the space of irrational numbers and the Hilbert cube.  However, most of the continua studied by continuum theorists are not homogeneous. Indeed, even the arc, which is considered the simplest non-degenerate continuum, is not homogeneous.  Nevertheless, even when they are not homogeneous, many of these continua share the common feature that, informally speaking, homeomorphisms can only distinguish a finite number of different points.

To formalize this notion, let $\mathcal{H}(X)$ denote the group of autohomeomorphisms of the space $X$. Notice that $\mathcal{H}(X)$ acts on $X$ by $\langle h,x \rangle \mapsto h(x)$.  The orbit of a point $x \in X$ under this action is $\{h(x) \colon h \in \mathcal{H}(X)\}$.  The \emph{homogeneity degree} of $X$ is defined to be the number of these orbits.  It turns out that many relevant examples of spaces in continuum theory have a finite degree of homogeneity.  See the discussion in the introduction of \cite{a-h-pj} for a comprehensive list of results about continua with low homogeneity degree.

An important class of continua is the class of dendroids.  Two interesting subclasses of dendroids are dendrites and smooth fans.  In \cite{a-pj} and \cite{a-h-pj} the authors gave complete lists of all dendrites with homogeneity degree $3$ and all smooth fans with homogeneity degree $3$ and $4$, respectively.
 
In this paper we will go to the exact opposite of homogeneity.  Recall that a topological space is called \emph{rigid} if its only autohomeomorphism is the identity.  Equivalently, $X$ is rigid if every orbit is a point.  When we say that a homeomorphism is trivial we mean that it is equal to the identity function.

For zero-dimensional separable metric spaces, it known from \cite{rigid-borel} that an infinite rigid space cannot be Borel (in any Polish space in which it is embedded) and in fact, as shown in \cite{rigid-analytic}, the statement ``there are infinite zero-dimensional rigid separable metric spaces that are analytic" cannot be decided from the axioms of \textsf{ZFC}.  However, for continua, rigidity is not uncommon.  Well-known examples of rigid continua are Cook continua \cite{cook}.  Of course Cook continua are complicated continua but we can make the following observations.

\begin{obs}
\label{rigid-dendrite}
There exists a rigid dendrite.
\end{obs}

To prove Observation \ref{rigid-dendrite} one may follow the construction of Wazewski's universal dendrite (see, for example, \cite[section 10.37]{nadler}): it is possible to construct a dendrite with a dense set of ramification points such that any two ramification points have different orders.  Then any autohomeomorphism is trivial when restricted to the dense set of ramification points, so it is trivial.

The next observation concerns smooth fans; see Section \ref{sec:smooth fans} below for the relevant definitions.

\begin{obs}
There is no smooth fan that is rigid.
\end{obs}

\begin{proof}
%
For a smooth fan $X$, consider the set $E(X)$ of endpoints of $X$. If $E(X)$ is dense in $X$, then $X$ is homeomorphic to the Lelek fan according to \cite{charatonik-lelek}; as observed in \cite{a-h-pj} in this case $E(X)$ is an uncountable orbit.  Now, assume that $E(X)$ is not dense in $X$.  We may assume that $X$ is a subset of the Cantor fan (see \cite{charatonik-fans} and \cite{eberhart}), so that if we remove the vertex of $X$ from $X$, the result is a subset of $C \times (0,1]$, where $C$ is the Cantor set.  Since $E(X)$ is not dense in $X$, there exist a clopen set $U \subset C$ and $0 < a < b < 1$ such that $(U \times [a,b]) \cap X \neq \emptyset$ and $(U \times [a,b]) \cap E(X) = \emptyset$.  Let $f\colon [a,b] \to [a,b]$ be any non-trivial homeomorphism with $f(a) = a$ and $f(b) = b$, and let $g\colon U \times[a,b] \to U \times [a,b]$ be defined by $g(\langle p,t\rangle)=\langle p,f(t)\rangle$.  Define $H \colon X \to X$ by $H(x)=x$ if $x\notin U \times [a,b]$ and $H(x)=g(x)$ if $x\in U \times [a,b]$. Then $H$ is a non-trivial homeomorphism.
\end{proof}

So it would seem that there is nothing remarkable about smooth fans and rigidity.  However, in this paper we would like to argue that this is not the case.  We cannot reach rigidity for smooth fans but we can reach it in the endpoints in the sense that there exist smooth fans $X$ such that every homeomorphism $h \colon X \to X$ is such that $h$ is trivial when restricted to the set of endpoints of $X$.  A smooth fan with this property will be called \emph{endpoint rigid}.

In Section \ref{sec:endpoint rigid}, we will produce several examples of endpoint rigid smooth fans, beginning with Proposition \ref{prop:K-fan exists}, in which such fans are constructed whose set of endpoints is homeomorphic to (1) the natural numbers, (2) the irrationals, and (3) the product of the Cantor set with the natural numbers.  Notice that each of these three endpoint sets is homogeneous.  So even if the set of endpoints of a smooth fan is homogeneous, in some sense this homogeneity is completely lost by the way it is placed inside the smooth fan.

It remains an open problem (see Questions \ref{open-1}, \ref{open-2} and \ref{open-3}) exactly which spaces may be obtained as the set of endpoints of an endpoint rigid smooth fan.  Motivated by this question, we show in Section \ref{sec:Lelek fan characterization} that the set of endpoints of an endpoint rigid smooth fan cannot be homeomorphic to complete Erd\H{o}s space (defined in Section \ref{sec:endpoint sets} below).  In fact, we prove in Theorem \ref{thm:E(X) not dense} below that the Lelek fan is the only smooth fan whose set of endpoints is homeomorphic to complete Erd\H{o}s space, thus obtaining a new characterization of the Lelek fan (see Corollary \ref{coro:Lelek fan characterization}).

\section{Preliminaries}

All the spaces considered in this paper are separable and metric.  A \emph{continuum} is a compact, connected metric space.  A continuum is \emph{non-degenerate} if it contains more than one point.  An \emph{arc} is a continuum which is homeomorphic to the interval $[0,1]$.

We denote the natural numbers by $\omega = \{0,1,2,\ldots\}$.  Given a point $x$ in an arcwise connected continuum $X$ and a natural number $n \in \omega\smallsetminus\{0\}$, we say the \emph{order} of $x$ is at least $n$ if there exist $n$ arcs $A_1,\ldots,A_n \subset X$, all of which have $x$ as one endpoint, and which are otherwise pairwise disjoint.  We say $x$ has order (exactly) $n$ if the order of $x$ is at least $n$ but not at least $n+1$.

Suppose $X$ is a continuum in which for any two distinct points $x,y \in X$, there is a unique arc in $X$ whose endpoints are $x$ and $y$.  In such a case, we denote this arc by $xy$.  In particular, we will use this notation with (smooth) fans throughout this paper.

\subsection{Smooth fans}
\label{sec:smooth fans}

Let $C$ denote the Cantor set.  The \emph{Cantor fan} $F_C$ is defined to be the cone over the Cantor set.  A \emph{smooth fan} is any non-degenerate continuum which is homeomorphic to a subcontinuum of the Cantor fan (see \cite{charatonik-fans} and \cite{eberhart}).  We will always assume that an embedding of a smooth fan $X$ in the Cantor fan $F_C$ contains the cone point of $F_C$, and that this point has order at least $2$ in $X$.  The point of $X$ corresponding to the cone point of $F_C$ is called the \emph{vertex} of $X$.  It is uniquely defined if it has order at least $3$.  A point $e \in X$ is an \emph{endpoint} if it has order $1$.  We denote by $E(X)$ the set of all endpoints of $X$.

The \emph{Lelek fan}, constructed in \cite{lelek}, is the unique smooth fan $X$ with the property that $E(X)$ is dense in $X$ \cite{charatonik-lelek} (see also Corollary \ref{coro:Lelek fan characterization} below).

\subsection{USC functions}
\label{sec:USC functions}

For the sake of explicit constructions of smooth fans, it will be convenient to make use of the language of upper semi-continuous functions defined on the Cantor set.  Recall that a function $\varphi \colon C \to [0,\infty)$ is \emph{upper semi-continuous} (\emph{USC}) if for every $t \in (0,\infty)$ the set $\varphi^\leftarrow[(-\infty,t)]$ is open.  Following \cite{dijkstra-vm-memAMS}, we define
\begin{align*}
G_0^\varphi &= \{\langle x,\varphi(x) \rangle \colon x \in C, \, \varphi(x) > 0\}, \textrm{ and} \\
L_0^\varphi &= \{\langle x,y \rangle \colon x \in C, \, 0 \leq y \leq \varphi(x)\}.
\end{align*}
Observe that $\varphi$ is USC if and only if $L_0^\varphi$ is closed in $C \times \mathbb{R}$.

A USC function $\varphi \colon C \to [0,\infty)$ will be called \emph{non-degenerate} if there are at least two distinct points $x_1,x_2 \in C$ such that $\varphi(x_1) > 0$ and $\varphi(x_2) > 0$.

\begin{lemma}
\label{lemma:smooth fan USC}
Given any non-degenerate USC function $\varphi \colon C \to [0,\infty)$ defined on the Cantor set $C$, the space $L_0^\varphi / (C \times \{0\})$ is a smooth fan whose endpoint set is the image of $G_0^\varphi$ under the quotient projection.  Conversely, given any smooth fan $X$, there exists a non-degenerate USC function $\varphi \colon C \to [0,1]$ defined on the Cantor set $C$ such that $L_0^\varphi / (C \times \{0\}) \approx X$ and $G_0^\varphi \approx E(X)$.
\end{lemma}

\begin{proof}
The forward direction is clear, and is left to the reader.

For the converse, let us construct an explicit embedding of the Cantor fan $F_C$ in $\mathbb{R}^2$: namely, fix a Cantor set $C \subset \mathbb{R}$, let $p = \langle 0,0 \rangle$, and let $F_C$ be the union of all straight-line segments from $p$ to points of $C \times \{1\}$.

Let $X$ be a smooth fan, and assume that $X \subseteq F_C$, and that $p$ is the vertex of $X$ (and the order of $p$ in $X$ is at least $2$).  Define $\varphi \colon C \to [0,1]$ as follows: given $x \in C$, let $\varphi(x)$ be the maximum $y$-coordinate of a point of $X$ on the straight-line segment from $p$ to $\langle x,1 \rangle$.  We leave it to the reader to confirm that this function $\varphi$ satisfies the desired properties.
\end{proof}

\section{Endpoint rigid fans}
\label{sec:endpoint rigid}

\begin{defi}
A smooth fan $X$ will be called \emph{endpoint rigid} if for each homeomorphism $h \colon X \to X$, $h(x) = x$ for all $x \in E(X)$.
\end{defi}

In order to construct endpoint rigid smooth fans, we will make use of the notion of a $\mathcal{K}$-fan, defined next.

Two subsets $A_1,A_2$ of $(0,1)$ will be called \emph{order isomorphic} if there is an order preserving autohomeomorphism $h \colon [0,1] \to [0,1]$ such that $h(A_1) = A_2$.

\begin{defi}
\label{defi:K-fan}
Let $\mathcal{K} = \{K_n \colon n \in \omega\}$ be a countably infinite set of compact subsets of $(0,1)$, no two of which are order isomorphic.  A smooth fan $X$ with vertex $p$ is a \emph{$\mathcal{K}$-fan} if there exists a subset $E_{\mathcal{K}} = \{e_n \colon n \in \omega\} \subseteq E(X)$, such that:
\begin{enumerate}[label=(\roman*)]
    \item \label{en legs} For each $n \in \omega$, $\overline{E(X)} \cap pe_n \smallsetminus \{p,e_n\}$ is order isomorphic to $K_n$;
    \item \label{non-en legs} For each $e \in E(X) \smallsetminus E_{\mathcal{K}}$, $\overline{E(X)} \cap pe \smallsetminus \{p,e\}$ is not order isomorphic to any element of $\mathcal{K}$; and
    \item \label{E_K dense} For each $e \in E(X) \smallsetminus E_{\mathcal{K}}$, $\overline{E_{\mathcal{K}}} \cap pe \smallsetminus \{p\} \neq \emptyset$.
\end{enumerate}
\end{defi}

\begin{prop}
\label{prop:K-fan rigid}
Let $\mathcal{K}$ be a collection as in Definition \ref{defi:K-fan}.  Any $\mathcal{K}$-fan is endpoint rigid.
\end{prop}

\begin{proof}
Let $X$ be a $\mathcal{K}$-fan with vertex $p$, and let $E_{\mathcal{K}} = \{e_n \colon n \in \omega\}$ be as in Definition \ref{defi:K-fan}.  Let $h \colon X \to X$ be a homeomorphism. It follows from \ref{en legs} and \ref{non-en legs} of Definition \ref{defi:K-fan} that $E_{\mathcal{K}}$ is invariant under $h$.  Now, since no two of the spaces in $\mathcal{K}$ are order isomorphic, by \ref{en legs} of Definition \ref{defi:K-fan} we have that $h$ is the identity on $E_{\mathcal{K}}$.  From \ref{E_K dense} of Definition \ref{defi:K-fan}, we conclude that $h$ fixes at least one point other than $p$ on each leg of $X$, hence it is the identity on $E(X)$.  Thus $X$ is endpoint rigid.
\end{proof}

\begin{prop}
\label{prop:K-fan exists}
Let $\mathcal{K}$ be a collection as in Definition \ref{defi:K-fan}.  There exists a $\mathcal{K}$-fan $X$. Moreover, if $E_{\mathcal{K}} = \{e_n \colon n \in \omega\}$ is as in Definition \ref{defi:K-fan} and $p$ is the vertex of $X$, then $X$ can be constructed so that either:
\begin{enumerate}[label=(\arabic*)]
    \item (Assuming $\emptyset \notin \mathcal{K}$) $E(X)$ is homeomorphic to the natural numbers and $E(X) = E_{\mathcal{K}}$; or
    \item $E(X)$ is homeomorphic to the irrationals, and for each $e \in E(X) \smallsetminus E_{\mathcal{K}}$, $\overline{E(X)} \cap pe = \{e\}$; or
    \item $E(X)$ is homeomorphic to the product of the Cantor set with the natural numbers, and for each $e \in E(X) \smallsetminus E_{\mathcal{K}}$, $\overline{E(X)} \cap pe = \{p,e\}$.
\end{enumerate}
\end{prop}

\begin{proof}
To begin, we recursively construct a particular Cantor set $C \subset [0,1]$.  The set $C$ will be indexed by the set $\omega^{\leq \omega}$ of countable sequences of natural numbers (possibly finite or even empty). 

Recall that $\omega^{\leq \omega} = \bigcup \{\omega^\alpha \colon \alpha \leq \omega\}$, that is, the set of functions with domain either a natural number or $\omega$, and image contained in $\omega$.  If $s \in \omega^\alpha$ then $\alpha$ is the \emph{length} of $s$ and will be denoted by $|s|$.  Given $s \in \omega^{\leq \omega}$ and $i < |s|$ then the $i$-th term of $s$ is $s(i)$.  If $s \in \omega^{\leq \omega}$ and $n \leq |s|$ then $s \restriction n$ denotes the restriction of $s$ to $n$.  We will also use the notation $\omega^{< \omega} = \bigcup \{\omega^n \colon n < \omega\}$.  If $n < \omega$, $s \in \omega^n$ and $i \in \omega$ then $s^\frown i$ denotes the sequence such that $(s^\frown i) \restriction n = s$ and $(s^\frown i)(n) = i$.  Given $s \in \omega^{< \omega}$, its \emph{immediate successors} are the sequences $s^\frown i$ for $i \in \omega$.

We start by recursively defining a countable sequence of points $\{x_s \colon s \in \omega^{< \omega}\} \subset [0,1]$ and a countable sequence of closed intervals $\{I_s \colon s \in \omega^{< \omega}\}$ with the following properties:
\begin{enumerate}[label=(\roman*)]
    \item $x_\emptyset = 0$ and $I_\emptyset = [0,1]$;
    \item For every $s \in \omega^{< \omega}$, $\{x_{s^\frown n} \colon n \in \omega\}$ is a decreasing sequence of real numbers contained in $\inte[\mathbb{R}]{I_s}$ converging to $x_s$; and
    \item Given $s \in \omega^{< \omega}$, assume that $I_s = [x_s,b]$.  Then $I_{s^\frown 0} = [x_{s^\frown 0},\frac{1}{2}(x_{s^\frown 0} + b)]$ and $I_{s^\frown i} = [x_{s^\frown (i+1)},\frac{1}{2}(x_{s^\frown (i+1)} + x_{s^\frown i})]$ for every $i \in \omega$.
\end{enumerate}

See Figure \ref{fig:xs} for an illustration of an example of such points $x_s$ and intervals $I_s$ as described above.

\begin{figure}
    \centering
    \includegraphics[width=4.9in]{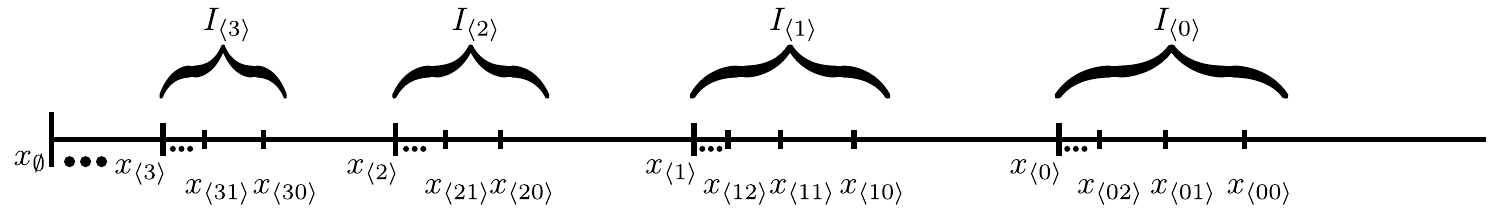}
    \caption{An illustration of a few of the points $x_s$ and intervals $I_s$ from the proof of Proposition \ref{prop:K-fan exists}.}
    \label{fig:xs}
\end{figure}

Now, let $f \in \omega^\omega$.  Notice that the sequence $\{x_{f \restriction n} \colon n \in \omega\}$ is an increasing sequence bounded by $1$; let $x_f$ be the limit of this sequence.  So define $C = \{x_s \colon s \in \omega^{\leq \omega}\}$.  

We argue that $C$ is a Cantor set. First, by our construction, $C$ does not have isolated points. Next, notice that $\{x_s\colon s\in \omega^{<\omega}\}$ is closed under decreasing sequences and every increasing sequence of elements of $\{x_s\colon s\in \omega^{<\omega}\}$ is of the form $x_f$ for some $f\in\omega^{\leq\omega}$. From this it follows that no element of $[0,1]\setminus C$ is the limit of a sequence of elements of $\{x_s\colon s\in \omega^{<\omega}\}$. Then $C$ is closed in $[0,1]$ and thus, compact. Finally, $C$ is zero-dimensional since the collection $\{I_s\cap C\colon s\in\omega^{<\omega}\}$ is easily seen to be a basis of clopen sets of $C$. From this, we conclude that $C$ is compact, metrizable, zero-dimensional and has no isolated points; thus, it is a Cantor set.

We make two remarks regarding convergence in $C$: given a sequence $s_1,s_2,\ldots$ in $\omega^{\leq \omega}$,
\begin{enumerate}[label=(\alph*)]
    \item $x_{s_k} \to x_f$ for some $f \in \omega^\omega$ if and only if $\max\{n \colon f \restriction n = s_k \restriction n\} \to \infty$ as $k \to \infty$; and
    \item $x_{s_k} \to x_s$ for some $s \in \omega^{< \omega}$ if and only if for all but finitely many $k$, either $s_k = s$, or $s_k$ extends $s$, and for such $k$ in the latter case, $s_k(|s|) \to \infty$ as $k \to \infty$.
\end{enumerate}

Fix a bijection $\theta \colon \omega^{< \omega} \to \omega$ with the property that
\begin{enumerate}[resume*]
    \item \label{theta monotone} $\theta(s_1) < \theta(s_2)$ whenever $s_2$ extends $s_1$.
\end{enumerate}
It follows that:
\begin{enumerate}[resume*]
    \item \label{theta estimate} If $s_1,s_2,\ldots \in \omega^{< \omega}$ are pairwise distinct, then $\theta(s_k) \to \infty$ as $k \to \infty$.
\end{enumerate}

We next define an upper semi-continuous function $\varphi \colon C \to [0,1]$.  We begin by recursively defining $\varphi(x_s)$ for each $s \in \omega^{< \omega}$. Let $\varphi(x_\emptyset) = 1$.

Let $s \in \omega^{< \omega}$, and suppose $\varphi(x_s) > 0$ has been defined.  We choose a set $Y_s \subset (0,\varphi(x_s))$ which is an order isomorphic copy of $K_{\theta(s)}$, and choose values $\varphi(x_{s^\frown i})$ for $i < \omega$, with details depending on which of the properties (1), (2), or (3) we wish to achieve:
\begin{itemize}
    \item For (1), we choose $Y_s \subset (0,\min\{\varphi(x_s),2^{-\theta(s)}\})$, and for $i = 0,1,\ldots$ we choose values $\varphi(x_{s^\frown i}) \in Y_s$ in such a way that $\overline{\{\varphi(x_{s^\frown i}) \colon i \geq n\}} = Y_s$ for each $n \in \omega$.
    
    \item For (2), we choose $Y_s \subset (\max\{0,\varphi(x_s) - 2^{-\theta(s)}\},\varphi(x_s))$, and for $i = 0,1,\ldots$ we choose values $\varphi(x_{s^\frown i}) \in Y_s \cup \{\varphi(x_s)\}$ in such a way that $\overline{\{\varphi(x_{s^\frown i}) \colon i \geq n\}} = Y_s \cup \{\varphi(x_s)\}$ for each $n \in \omega$.
    
    \item For (3), we choose $Y_s \subset (0,\min\{\varphi(x_s),2^{-\theta(s)}\})$, and for $i = 0,1,\ldots$ we choose values $\varphi(x_{s^\frown i}) \in Y_s \cup \{\varphi(x_s)\}$ in such a way that $\overline{\{\varphi(x_{s^\frown i}) \colon i \geq n\}} = Y_s \cup \{\varphi(x_s)\}$ for each $n \in \omega$.
\end{itemize}

See Figure \ref{fig:phis} for an illustration in each of these three cases.

\begin{figure}
    \centering
    \includegraphics[width=4.9in]{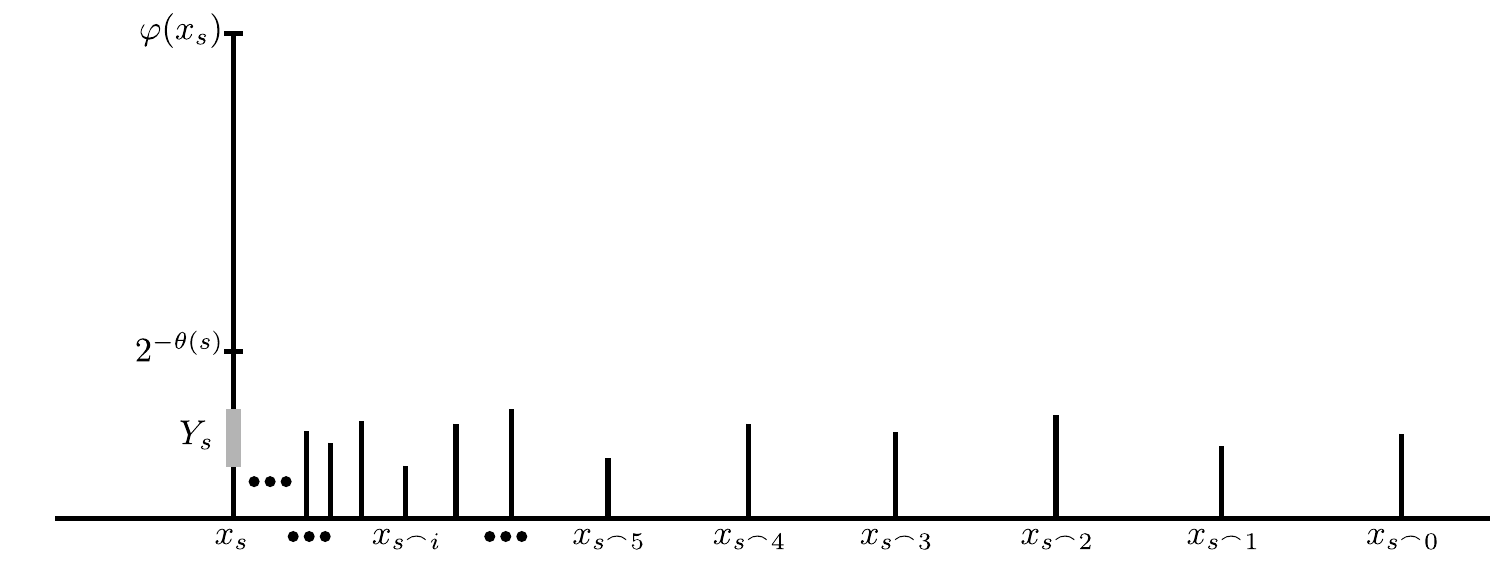}
    \includegraphics[width=4.9in]{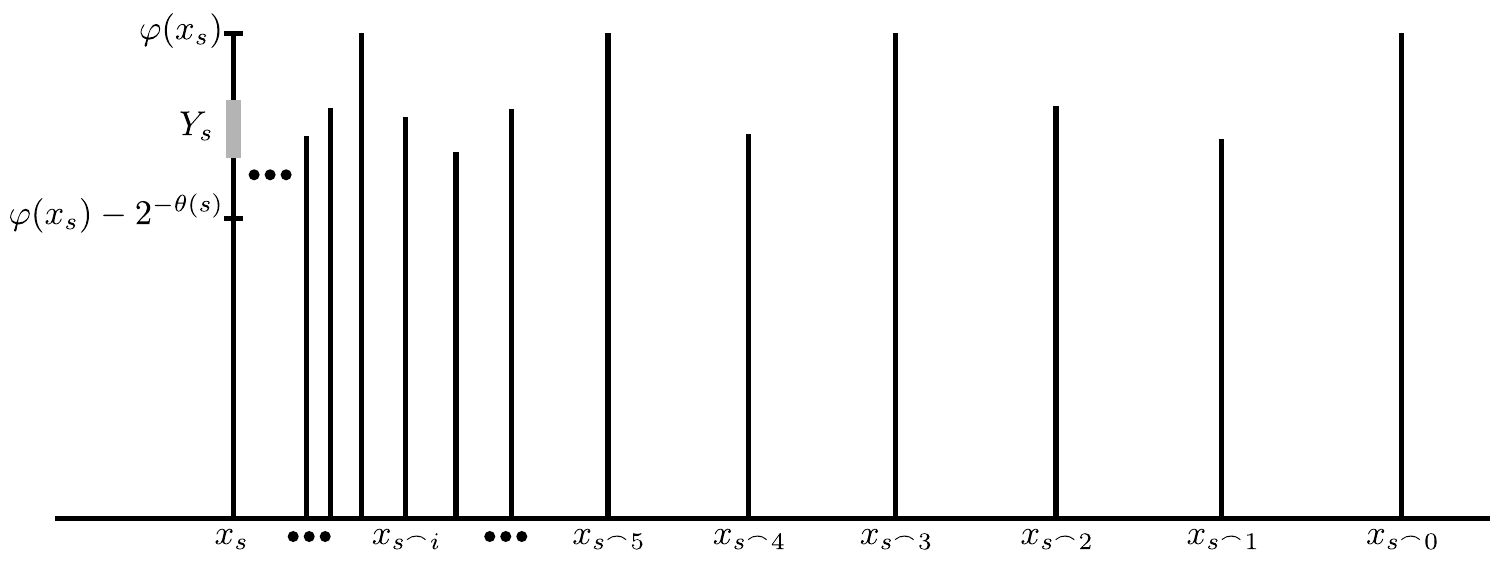}
    \includegraphics[width=4.9in]{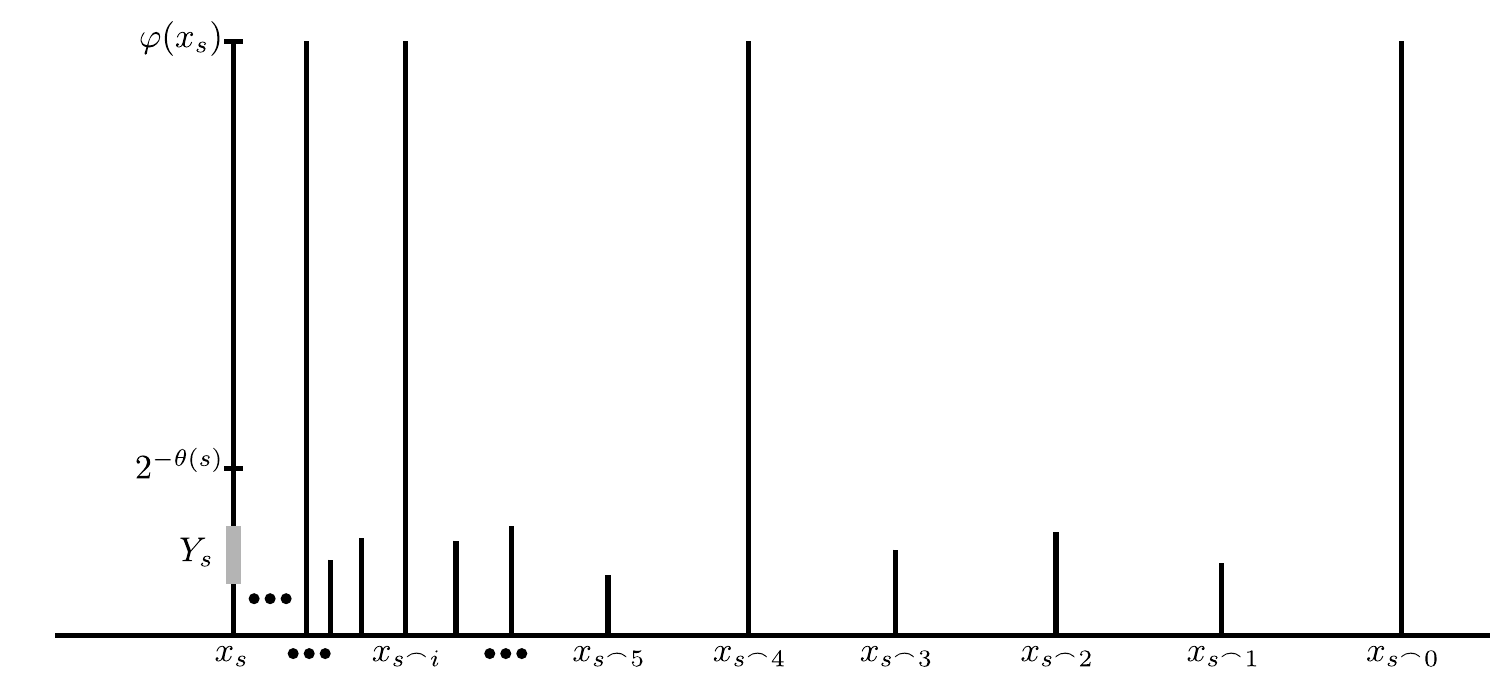}
    \caption{Illustration of the definition of $\varphi$, in cases (1), (2), and (3), respectively.  The height of the vertical line at $x_s$ represents the value of $\varphi(x_s)$.}
    \label{fig:phis}
\end{figure}

We remark that by construction,
\begin{enumerate}[resume*]
\item \label{phi monotone} $\varphi(x_{s_1}) \geq \varphi(x_{s_2})$ whenever $s_2$ extends $s_1$.
\end{enumerate}
We make further observations about $\varphi$ depending on the variant (1), (2), or (3) of the construction:
\begin{enumerate}[resume, label={}]
    \item \begin{enumerate}[label=(\alph{enumi}\arabic*), leftmargin=0pt, parsep=5pt]
        \item \label{phi estimate 1} For (1), for any $\emptyset \neq s \in \omega^{< \omega}$, $\varphi(x_s) < 2^{-\theta(s \restriction |s|-1)}$.
        
        This holds because for any $s \in \omega^{< \omega}$ and any $i \in \omega$, $\varphi(x_{s^\frown i}) \in Y_s$, hence by choice of $Y_s$, $\varphi(x_{s^\frown i}) < 2^{-\theta(s)}$.
        
        \item \label{phi estimate 2} For (2), for any $s_1,s_2 \in \omega^{< \omega}$ such that $s_2$ extends $s_1$, $\varphi(x_{s_1}) - \varphi(x_{s_2}) < 2^{-\theta(s_1)+1}$.
        
        To prove this, note that for any $s \in \omega^{< \omega}$ and any $i \in \omega$, $\varphi(x_{s^\frown i}) \in Y_s \cup \{\varphi(x_s)\}$, hence by choice of $Y_s$, $\varphi(x_{s^\frown i}) > \varphi(x_s) - 2^{-\theta(s)}$.  So
        \begin{align*}
            \varphi(x_{s_1}) - \varphi(x_{s_2}) &= \sum_{n=|s_1|}^{|s_2|-1} \varphi(x_{s_2 \restriction n}) - \varphi(x_{s_2 \restriction n+1}) \\
            & < \sum_{n=|s_1|}^{|s_2|-1} 2^{-\theta(s_2 \restriction n)} \\
            & \leq \sum_{m=0}^{|s_2|-|s_1|-1} 2^{-\theta(s_1)-m}  \textrm{\quad by \ref{theta monotone}} \\
            & < 2^{-\theta(s_1)+1} .
        \end{align*}
        
        \item \label{phi estimate 3} For (3), for any $\emptyset \neq s \in \omega^{< \omega}$, either $\varphi(x_s) < 2^{-\theta(s \restriction |s|-1)}$ or $\varphi(x_s) = \varphi(x_{s \restriction |s|-1})$.
        
        This holds because for any $s \in \omega^{< \omega}$ and any $i \in \omega$, either $\varphi(x_{s^\frown i}) \in Y_s$ or $\varphi(x_{s^\frown i}) = \varphi(x_s)$.  In the former case, the estimate follows by the choice of $Y_s$ as in (1).
    \end{enumerate}
\end{enumerate}

Finally, for $f \in \omega^\omega$, define $\varphi(x_f) = \lim_{n \to \infty} \varphi(x_{f \restriction n})$; this limit exists by \ref{phi monotone}.

\begin{claim}
\label{claim:phi usc}
$\varphi$ is upper semi-continuous.
\end{claim}

\begin{proof}[Proof of Claim \ref{claim:phi usc}]
\renewcommand{\qedsymbol}{\textsquare (Claim \ref{claim:phi usc})}
Suppose $s_1,s_2,\ldots \in \omega^{\leq \omega}$ are such that $x_{s_k}$ converges in $C$ and $\varphi(x_{s_k})$ converges to $y \in [0,1]$.

If $x_{s_k} \to x_f$ for some $f \in \omega^\omega$, then for each $k$ let $n(k)$ be maximal such that $s_k \restriction n(k) = f \restriction n(k)$.  Then $n(k) \to \infty$ as $k \to \infty$.  By \ref{phi monotone} we have $\varphi(x_{s_k}) \leq \varphi(x_{s_k \restriction n(k)}) = \varphi(x_{f \restriction n(k)})$.  The latter quantity converges to $\varphi(x_f)$ by definition of $\varphi$.  Thus $y \leq \varphi(x_f)$ in this case.

If $x_{s_k} \to x_s$ for some $s \in \omega^{< \omega}$, then for all but finitely many $k$, either $s_k = s$ or $s_k$ extends $s$, and for all such $k$ we have $\varphi(x_{s_k}) \leq \varphi(x_s)$ by \ref{phi monotone}.  Thus $y \leq \varphi(x_s)$ in this case as well.
\end{proof}

Let $X = L_0^\varphi / (C \times \{0\})$, and let $p$ represent the point in the quotient corresponding to $C \times \{0\}$.  Then $X$ is a smooth fan with vertex $p$.

For points $\langle x_s,y \rangle \in L_0^\varphi$ with $y > 0$, for convenience we identify $\langle x_s,y \rangle$ with its image in $X$ under the quotient projection.  The set of endpoints $E(X)$ is thus the set $\{\langle x_s,\varphi(x_s) \rangle \colon s \in \omega^{\leq \omega}$ and $\varphi(s) > 0\}$, and given any endpoint $e = \langle x_s,\varphi(x_s) \rangle$ in $X$, the set $pe \smallsetminus \{p\}$ is identified with $\{x_s\} \times (0,\varphi(x_s)]$.

For each $s \in \omega^{< \omega}$, let $e_{\theta(s)} \in X$ be the endpoint $\langle x_s,\varphi(x_s) \rangle$, and let $E_{\mathcal{K}} = \{e_{\theta(s)} \colon s \in \omega^{< \omega}\}$.  For any $f \in \omega^\omega$ such that $\varphi(x_f) > 0$, it is immediate from the definition of $\varphi(x_f)$ that the endpoint $\langle x_f,\varphi(x_f) \rangle$ is in $\overline{E_{\mathcal{K}}}$.  This already confirms \ref{E_K dense} of Definition \ref{defi:K-fan}.  Moreover, we have $\overline{E(X)} = \overline{E_{\mathcal{K}}}$.

\begin{claim}
\label{claim:endpoint trace s}
For each $s \in \omega^{< \omega}$, $\overline{E_{\mathcal{K}}} \cap \left[ \{x_s\} \times (0,\varphi(x_s)) \right] = \{x_s\} \times Y_s$.
\end{claim}

\begin{proof}[Proof of Claim \ref{claim:endpoint trace s}]
\renewcommand{\qedsymbol}{\textsquare (Claim \ref{claim:endpoint trace s})}
Let $s \in \omega^{< \omega}$.  It is clear from the construction that $\{x_s\} \times Y_s \subseteq \overline{\{\langle x_{s^\frown i},\varphi(x_{s^\frown i}) \rangle \colon i \in \omega\}} \subseteq \overline{E_{\mathcal{K}}}$.

For the reverse inclusion, let $s_1,s_2,\ldots$ belong to $\omega^{< \omega}$, and suppose that $x_{s_k} \to x_s$ and $\varphi(x_{s_k}) \to y$ as $k \to \infty$.  We must prove that $y \in Y_s \cup \{0,\varphi(x_s)\}$.

We may assume that $s_k$ extends $s$ for all $k$ and $s_k(|s|) \to \infty$ as $k \to \infty$.  This implies $\theta(s_k \restriction |s|+1) \to \infty$ as $k \to \infty$ by \ref{theta estimate}.  If $|s_k| = |s| + 1$ for infinitely many $k$, then $y \in Y_s \cup \{\varphi(x_s)\}$ by construction.  Assume therefore that (without loss of generality) $|s_k| \geq |s| + 2$ for all $k$.

\begin{itemize}
    \item For (1), observe that by \ref{phi monotone}, and \ref{phi estimate 1},
    \[ \varphi(x_{s_k}) \leq \varphi(x_{s_k \restriction |s|+2}) < 2^{-\theta(s_k \restriction |s|+1)} \to 0  \textrm{ as } k \to \infty .\]
    Therefore $y = 0$.

    \item For (2), by \ref{phi estimate 2} we have $\varphi(x_{s_k \restriction |s|+1}) - \varphi(s_k) < 2^{-\theta(s_k \restriction |s|+1)+1}$.  Since $\varphi(x_{s_k \restriction |s|+1}) \in Y_s \cup \{\varphi(x_s)\}$ for each $k$ and $2^{-\theta(s_k \restriction |s|+1)+1} \to 0$ as $k \to \infty$, we have $y \in Y_s \cup \{\varphi(x_s)\}$.
    
    \item For (3), if for infinitely many $k$ there exists $|s| + 2 \leq n \leq |s_k|$ such that $\varphi(x_{s_k \restriction n}) \in Y_{s_k \restriction n-1}$, then similarly as with (1) we have that $\varphi(x_{s_k}) \to 0$ as $k \to \infty$, so $y=0$.  Assume then that (without loss of generality) for all $k$, for each $|s| + 2 \leq n \leq |s_k|$, $\varphi(x_{s_k \restriction n}) = \varphi(x_{s_k \restriction n-1})$.  It follows that $\varphi(x_{s_k}) = \varphi(x_{s_k \restriction |s|+1}) \in Y_s \cup \{\varphi(x_s)\}$ for all $k$.  Thus $y \in Y_s \cup \{\varphi(x_s)\}$.
\end{itemize}
\end{proof}

\begin{claim}
\label{claim:endpoint trace f}
Let $f \in \omega^{\omega}$.
\begin{itemize}
    \item For (1), $\varphi(x_f) = 0$.
    \item For (2), if $\varphi(x_f) > 0$ then $\overline{E_{\mathcal{K}}} \cap \left[ \{x_f\} \times [0,\varphi(x_f)] \right] = \{\langle x_f,\varphi(x_f) \rangle\}$.
    \item For (3), if $\varphi(x_f) > 0$ then $\overline{E_{\mathcal{K}}} \cap \left[ \{x_f\} \times [0,\varphi(x_f)] \right] = \{\langle x_f,0 \rangle, \langle x_f,\varphi(x_f) \rangle\}$.
\end{itemize}
\end{claim}

\begin{proof}[Proof of Claim \ref{claim:endpoint trace f}]
\renewcommand{\qedsymbol}{\textsquare (Claim \ref{claim:endpoint trace f})}
\
\begin{itemize}
    \item For (1), observe that for each $n \geq 1$, by construction $\varphi(x_{f \restriction n}) \in Y_{f \restriction n-1} \subset (0,2^{-\theta(f \restriction n-1)})$.  Hence $\varphi(x_f) = \lim_{n \to \infty} \varphi(x_{f \restriction n}) = 0$ because $\theta(f \restriction n-1) \to \infty$ as $n \to \infty$ by \ref{theta estimate}.
\end{itemize}

For (2) and (3), assume now that $\varphi(x_f) > 0$.  We have already shown above that $\langle x_f,\varphi(x_f) \rangle$ is in $\overline{E_{\mathcal{K}}}$.

\begin{itemize}
    \item For (3), we also need to show that $\langle x_f,0 \rangle \in \overline{E_{\mathcal{K}}}$.  To see this, we construct a sequence as follows: for each $n \in \omega$, choose $i(n) \in \omega$ such that $\varphi(x_{(f \restriction n)^\frown i(n)}) \in Y_{f \restriction n} \subseteq (0,2^{-\theta(f \restriction n)})$.  Observe that $x_{(f \restriction n)^\frown i(n)} \to x_f$ as $n \to \infty$.  Since $\theta(f \restriction n) \to \infty$ as $n \to \infty$ by \ref{theta estimate}, it follows that $\langle x_f,0 \rangle \in \overline{\{\langle x_{(f \restriction n)^\frown i(n)},\varphi(x_{(f \restriction n)^\frown i(n)}) \rangle \colon n \in \omega\}} \subseteq \overline{E_{\mathcal{K}}}$.
\end{itemize}

For the reverse inclusions, suppose $s_1,s_2,\ldots$ belong to $\omega^{< \omega}$, and suppose that $x_{s_k} \to x_f$ and $\varphi(x_{s_k}) \to y$ as $k \to \infty$.  For each $k$, let $n(k) = \max\{n \in \omega \colon f \restriction n = s_k \restriction n\}$.  Then $n(k) \to \infty$ as $k \to \infty$.

\begin{itemize}
    \item For (2) we must prove that $y = \varphi(x_f)$.  By \ref{theta estimate} and \ref{phi estimate 2}, $\varphi(x_{f \restriction n(k)}) - \varphi(x_{s_k}) < 2^{-\theta(f \restriction n(k))+1} \to 0$ as $k \to \infty$.  Thus
    \[ y = \lim_{k \to \infty} \varphi(x_{s_k}) = \lim_{k \to \infty} \varphi(x_{f \restriction n(k)}) = \varphi(x_f) .\]
    
    \item For (3) we must prove that $y \in \{0,\varphi(x_f)\}$.  First observe that there must exist some $n_0 \in \omega$ such that $\varphi(x_{f \restriction n}) = \varphi(x_{f \restriction n_0})$ for all $n \geq n_0$, since otherwise $\varphi(x_f)$ would be $0$ as in (1).  We may assume without loss of generality that $n(k) \geq n_0$ for all $k$.
    
    If there are infinitely many $k$ such that $\varphi(s_k \restriction m(k)) \in Y_{s_k \restriction m(k)-1}$ for some $n(k) < m(k) \leq |s_k|$, then for such $k$, by \ref{phi monotone} and \ref{phi estimate 3}, $\varphi(x_{s_k}) \leq \varphi(x_{s_k \restriction m(k)}) < 2^{-\theta(s_k \restriction m(k)-1)+1}$, and it follows from \ref{theta estimate} that $y = 0$ in this case.  Otherwise, we have $\varphi(s_k) = \varphi(x_{f \restriction n(k)}) = \varphi(x_f)$ for all but finitely many $k$, hence $y = \varphi(x_f)$.
\end{itemize}
\end{proof}

This completes the proof that $X$ is a $\mathcal{K}$-fan in all three constructions (1), (2), and (3).  Lastly, we identify the spaces $E(X)$ in each case:

\begin{itemize}
    \item For (1), it should be clear from the construction that $E(X) = E_{\mathcal{K}}$ is a countable discrete space.
    
    \item For (2), note that $E(X)$ is completely metrizable and separable.  Any neighborhood of an endpoint contains infinitely many endpoints of the form $\langle x_s,\varphi(x_s)\rangle$ where $s \in \omega^{< \omega}$, and in fact will contain the set $\{x_s\} \times Y_s$ for all but finitely many of them.  Since by construction there are sequences of endpoints converging to these sets $\{x_s\} \times Y_s$, we conclude that $E(X)$ is nowhere locally compact.

    To see that $E(X)$ is zero-dimensional, let $e \in E(X)$, and let $\varepsilon > 0$.  We consider two cases.

    In the first case, $e = \langle x_s,\varphi(x_s)\rangle$ for some $s \in \omega^{< \omega}$.  We may suppose $\varepsilon$ is small enough so that $Y_s \cap (\varphi(x_s) - \varepsilon, 1] = \emptyset$.  Let $U \subset C$ be a closed and open neighborhood of $x_s$ in $C$ which is small enough so that for each $f \in \omega^{\leq \omega}$ with $x_f \in U$, $f \restriction |s| = s$ and
    \[ \sum_{\substack{t \in \omega^{< \omega} \\ x_t \in U}} 2^{-\theta(x_t)} < \varepsilon .\]  It follows that for each $f \in \omega^{\leq \omega}$ with $x_f \in U$, either $\varphi(x_{f \restriction (|s|+1)}) \in Y_s$, in which case $\varphi(x_f) < \varphi(x_s) - \varepsilon$, or $\varphi(x_{f \restriction (|s|+1)}) = \varphi(x_s)$, in which case $\varphi(x_s) - \varepsilon < \varphi(x_f) \leq \varphi(x_s)$.  Thus, if we let $N = U \times (\varphi(x_s) - \varepsilon, \varphi(x_s)]$, then $N \cap E(X)$ is an arbitrarily small closed and open neighborhood of $e$ in $E(X)$.

    In the second case, $e = \langle x_f,\varphi(x_f)\rangle$ for some $f \in \omega^{\omega}$.  Let $n$ be large enough so that
    \[ \sum_{\substack{t \in \omega^{< \omega} \\ t \restriction n = f \restriction n}} 2^{-\theta(x_t)} < \varepsilon .\]
    Let $U \subset C$ be a closed and open neighborhood of $x_f$ in $C$ which is small enough so that for each $g \in \omega^{\leq \omega}$ with $x_g \in U$, $g \restriction n = f \restriction n$.  It follows that for each $g \in \omega^{\leq \omega}$ with $x_g \in U$, $\varphi(x_{f \restriction n}) - \varepsilon < \varphi(x_g) \leq \varphi(x_{f \restriction n})$.  Thus, if we let $N = U \times (\varphi(x_{f \restriction n}) - \varepsilon, \varphi(x_{f \restriction n})]$, then $N \cap E(X)$ is an arbitrarily small closed and open neighborhood of $e$ in $E(X)$.
    
    Therefore $E(X)$ is homeomorphic to the irrational numbers.
    
    \item For (3), again note that $E(X)$ is a separable metric space.  $E(X)$ is not compact, as the vertex of the fan belongs to the closure of $E(X)$.

    We argue that each endpoint has a neighborhood which is homeomorphic to the Cantor set.  Let $e \in E(X)$, and let $f \in \omega^{\leq \omega}$ be such that $e = \langle x_f,\varphi(x_f)\rangle$.  Note that by construction (3), since $Y_s \subset (0,2^{-\theta(s)})$ for each $s \in \omega^{< \omega}$, if $\varphi(x_{f \restriction (n+1)}) \in Y_{f \restriction n}$ for infinitely many $n$, then $\varphi(x_f) = 0$, which is a contradiction since $e \in E(X)$.  Therefore $\varphi(x_f) = \varphi(x_{f \restriction n})$ for some $n$.  Let $U \subset C$ be a closed and open neighborhood of $x_f$ in $C$, and let $\varepsilon > 0$, chosen small enough so that for each $s \in \omega^{< \omega}$ with $x_s \in U$, $s \restriction n = f \restriction n$ and $2^{-\theta(x_s)} < \varphi(x_f) - \varepsilon$.  Consider the neighborhood $N = U \times (\varphi(x_f) - \varepsilon, 1]$ of $e$.  Now\\\\
    \( N \cap E(X) = \{(x_g,\varphi(x_g)): g \in \omega^{\leq \omega} \textrm{, } g \restriction n = f \restriction n \textrm{, and } \varphi(x_g) = \varphi(x_f)\} .\)\\\\
    This latter set is a Cantor set, by the same argument that shows $C$ is a Cantor set.
    
    Thus $E(X)$ is locally compact and has no isolated points.  Therefore $E(X)$ is homeomorphic to the product of the Cantor set and the natural numbers.
\end{itemize}
\end{proof}

Clearly distinct families $\mathcal{K}$ lead to non-homeomorphic fans in Proposition \ref{prop:K-fan exists}.  Hence we conclude the following:

\begin{coro}
There is a family $\mathcal{F}$ of size $\mathfrak{c}$ consisting of pairwise non-homeomorphic endpoint rigid smooth fans.  Further, we can choose $\mathcal{F}$ in such a way that $E(X)$ is homeomorphic to the natural numbers for every $X \in \mathcal{F}$, to the irrationals for every $X \in \mathcal{F}$, or to the product of the Cantor set with the natural numbers for every $X \in \mathcal{F}$.
\end{coro}

Constructions (2) and (3) of Proposition \ref{prop:K-fan exists} yield examples of smooth fans whose endpoint sets have cardinality $\mathfrak{c}$.  Hence we have the following:

\begin{coro}
There exists a smooth fan with degree of homogeneity $\mathfrak{c}$.
\end{coro}

A straightforward modification of the constructions from Proposition \ref{prop:K-fan exists} yields the next two results.

\begin{prop}
\label{prop:K-fan exists generalization 1}
Let $X$ be a smooth fan with vertex $p$.  Let $\mathcal{K}$ be a collection as in Definition \ref{defi:K-fan}.  There exists a $\mathcal{K}$-fan $X'$ such that $X$ embeds in $X'$, and (treating $X$ as a subset of $X'$) $E(X) \subset E(X')$ and $X \subset \overline{E(X') \smallsetminus E(X)}$. Moreover, if $E_{\mathcal{K}} = \{e_n \colon n \in \omega\}$ is as in Definition \ref{defi:K-fan}, then $X'$ can be constructed so that either:
\begin{enumerate}[label=(\arabic*)]
    \item $E(X') \smallsetminus E(X)$ is homeomorphic to the natural numbers and $E(X') \smallsetminus E(X) = E_{\mathcal{K}}$; or
    \item $E(X') \smallsetminus E(X)$ is homeomorphic to the irrationals; or
    \item $E(X') \smallsetminus E(X)$ is homeomorphic to the product of the Cantor set with the natural numbers.
\end{enumerate}
\end{prop}

\begin{prop}
\label{prop:K-fan exists generalization 2}
Let $X$ be a smooth fan with vertex $p$, and assume that $p \notin \overline{E(X)}$.  Let $\mathcal{K}$ be a collection as in Definition \ref{defi:K-fan}.  There exists a $\mathcal{K}$-fan $X'$ such that $X$ embeds in $X'$, and (treating $X$ as a subset of $X'$) $E(X)$ is a clopen subset of $E(X')$. Moreover, if $E_{\mathcal{K}} = \{e_n \colon n \in \omega\}$ is as in Definition \ref{defi:K-fan}, then $X'$ can be constructed so that either:
\begin{enumerate}[label=(\arabic*)]
    \item $E(X') \smallsetminus E(X)$ is homeomorphic to the natural numbers and $E(X') \smallsetminus E(X) = E_{\mathcal{K}}$; or
    \item $E(X') \smallsetminus E(X)$ is homeomorphic to the irrationals; or
    \item $E(X') \smallsetminus E(X)$ is homeomorphic to the product of the Cantor set with the natural numbers.
\end{enumerate}
\end{prop}

To obtain these last two results, one can start with the fan $X$, then add more legs, one for each $s \in \omega^{\leq \omega} \smallsetminus \{\emptyset\}$, starting with a sequence of legs indexed by the sequences in $\omega^{< \omega}$ of length one, constructed in such a way that they converge to $X$ (for Proposition \ref{prop:K-fan exists generalization 1}) or to a subset of $X$ disjoint from $E(X)$ and whose intersection with each leg is not order isomorphic to any element of $\mathcal{K}$ (for Proposition \ref{prop:K-fan exists generalization 2}), then proceed with the rest of the construction as in Proposition \ref{prop:K-fan exists}.

We next introduce one simple construction on fans which we can use to generate more endpoint rigid fans.

\begin{defi}
Given finitely many smooth fans $X_1,\ldots,X_m$ whose vertices are $p_1,\ldots,p_m$, the \emph{sum} of these fans is
\[ \sum_{i=1}^m X_i = \left( \bigsqcup_{i=1}^m X_i \right) / \{p_1,\ldots,p_m\} .\]

Given countably infinitely many smooth fans $X_1,X_2,\ldots$ with vertices $p_1,p_2,\ldots$, consider the one point compactification $\bigsqcup_{i=1}^\infty X_i \cup \{\infty\}$ of $\bigsqcup_{i=1}^\infty X_i$.  The \emph{sum} of these fans is
\[ \sum_{i=1}^\infty X_i = \left(\bigsqcup_{i=1}^\infty X_i \cup \{\infty\} \right) / (\{p_1,p_2,\ldots\} \cup \{\infty\}) .\]
In both cases, observe that $\sum_i X_i$ is a smooth fan, and $E(\sum_i X_i) \approx \bigsqcup_i E(X_i)$.
\end{defi}

\begin{prop}
\label{prop:sum K-fan}
Suppose that $\mathcal{K}_1,\mathcal{K}_2,\ldots$ is a finite or countably infinite sequence of collections as in Definition \ref{defi:K-fan}, and suppose that no two of these collections contain sets which are order isomorphic to each other.  For each $i$, let $X_i$ be a $\mathcal{K}_i$-fan, and let $\mathcal{K} = \bigcup_i \mathcal{K}_i$.  Then $\sum_i X_i$ is a $\mathcal{K}$-fan.
\end{prop}

\section{Endpoint sets in smooth fans}
\label{sec:endpoint sets}

In this section we investigate which spaces may be realized as the set of endpoints of a smooth fan, or of an endpoint rigid smooth fan.  To start with, it is known from \cite[Theorem 7.5]{lelek} that the set of endpoints of any smooth fan is a Polish space.

\emph{Complete Erd\H{o}s space} is the subspace
\[ \Ec = \left\{ \langle x_n \rangle_{n \in \omega} \in \ell^2 \colon \forall i \in \omega, \, x_i \in \{0\} \cup \{1/n \colon n \in \mathbb{N}\} \right\} \]
of $\ell^2$, where $\ell^2$ is the Hilbert space of square-summable sequences of real numbers.  Complete Erd\H{o}s space was introduced by Erd\H{o}s in 1940 in \cite{er} as an example of totally disconnected and non-zero-dimensional space.  In \cite{k-o-t} it was proved that the Lelek fan has its set of endpoints homeomorphic to $\Ec$.

\begin{lemma}
\label{lemma:E(X) Ec}
For every smooth fan $X$ the set $E(X)$ is homeomorphic to a closed subset of $\Ec$.
\end{lemma}

\begin{proof}
Let $X$ be a smooth fan.  By Lemma \ref{lemma:smooth fan USC}, $E(X) \approx G_0^\varphi$ for some USC function $\varphi \colon C \to [0,1]$ defined on the Cantor set $C$.  Thus by \cite[Theorem 3.2]{DvM}, $E(X)$ is homeomorphic to a closed subset of $\Ec$.
\end{proof}

We are not aware of whether the converse implication of Lemma \ref{lemma:E(X) Ec} is true.  This would provide a very nice connection between the theory of almost zero-dimensional spaces and the theory of smooth fans.

\begin{ques}
\label{ques:AZD-smooth_fans}
Let $E$ be a closed subset of $\Ec$.  Is there a smooth fan $X$ such that $E(X)$ is homeomorphic to $E$?
\end{ques}

We now turn our attention to which spaces may be realized as the endpoint set of an endpoint rigid smooth fan.  The results of the previous section produce several examples of such spaces: for example the natural numbers, the irrationals, and the product of the Cantor set with the natural numbers (Propositions \ref{prop:K-fan rigid} and \ref{prop:K-fan exists}).  Propositions \ref{prop:K-fan exists generalization 1}, \ref{prop:K-fan exists generalization 2}, and \ref{prop:sum K-fan} (together with Proposition \ref{prop:K-fan rigid}) yield many more examples, such as the disjoint union of any compact zero-dimensional space with any one of the previously mentioned three spaces, disjoint unions of any of these examples, etc.

On the other hand, it is clear that no compact space may be the endpoint set of an endpoint rigid smooth fan.  Indeed, if $X$ is a smooth fan and $E(X)$ is compact, then $E(X)$ is zero-dimensional and $X$ is homeomorphic to the cone over $E(X)$, and one can easily find a non-trivial autohomeomorphism of $E(X)$, which naturally extends to a non-trivial autohomeomorphism of $X$.  Theorem \ref{thm:E(X) not dense} of the next section implies that $\Ec$ also cannot be realized as the endpoint set of an endpoint rigid smooth fan.  We leave the following questions:

\begin{ques}
\label{open-1}
Let $E$ be a non-compact closed subset of $\Ec$ that is not homeomorphic to complete Erd\H{o}s space.  Is there an endpoint rigid smooth fan $X$ such that $E(X)$ is homeomorphic to $E$?
\end{ques}

\begin{ques}
\label{open-2}
Let $E$ be a locally compact, non-compact, closed subset of $\Ec$.  Is there an endpoint rigid smooth fan $X$ such that $E(X)$ is homeomorphic to $E$?
\end{ques}

\begin{ques}
\label{open-3}
Let $E$ be a closed subset of $\Ec$ such that there is a unique smooth fan $X$ such that $E(X)$ is homeomorphic to $E$.  Is $E$ either compact or homeomorphic to complete Erd\H{o}s space?
\end{ques}

\section{Characterization of the Lelek fan}
\label{sec:Lelek fan characterization}

In this section we will prove that the only smooth fan whose set of endpoints is homeomorphic to $\Ec$ is the Lelek fan.

\begin{thm}
\label{thm:E(X) not dense}
If $X$ is a smooth fan and $E(X)$ is not dense in $X$, then $E(X)$ has a point of zero-dimensionality.  In particular, $E(X)$ is not homeomorphic to $\Ec$.
\end{thm}

By Lemma \ref{lemma:smooth fan USC}, Theorem \ref{thm:E(X) not dense} is implied by the following statement about USC functions, which we prove below.  Recall that $G_0^\varphi$ and $L_0^\varphi$ are defined in Section \ref{sec:USC functions}.

\begin{thm}
Let $\varphi \colon C \to [0,\infty)$ be an upper semi-continuous function defined on the Cantor set $C$.  Suppose there exists a non-empty open subset $U \subseteq C$ and $a_0 \in (0,\infty)$ such that $(U \times \{a_0\}) \cap L_0^\varphi$ is non-empty and disjoint from $G_0^\varphi$.  Then $G_0^\varphi$ has a point of zero-dimensionality.  In particular, $G_0^\varphi$ is not homeomorphic to $\Ec$.
\end{thm}

\begin{proof}
Fix a metric $d$ on $C$.  We will use the metric $\rho$ on $C \times [0,1]$ defined by
\[ \rho \left( \langle x_1,y_1 \rangle, \langle x_2,y_2 \rangle \right) = \max\{d(x_1,x_2), |y_1-y_2|\} .\]

To locate a point in $G_0^\varphi$ of zero-dimensionality, we will do a recursive construction in $\omega$ steps.  At step $n$, we will define a clopen set $V_n$ in $C$ and $a_n,b_n \in (0,\infty]$ with the following properties:
\begin{enumerate}[label=(\arabic*)]
    \item \label{recursion 1} $V_{n+1} \subseteq V_n$ for all $n \in \omega$;
    \item \label{recursion 2} $a_n \leq a_{n+1} < b_{n+1} \leq b_n$ for all $n \in \omega$;
    \item \label{recursion 3} $\emptyset \neq \varphi[V_n] \cap [a_0,\infty) \subset [a_n,b_n)$ for all $n \in \omega$; and
    \item \label{recursion 4} $\textrm{diam}(V_n) < 2^{-n}$ and $b_n - a_n \leq 2^{-n}$ for all $n \geq 1$.
\end{enumerate}

Let us give the construction.  In step $0$ we start with $a_0$ as in the hypotheses of the Theorem, let $V_0$ be a non-empty clopen subset of $U$ such that $(V_0 \times \{a_0\}) \cap L_0^\varphi \neq \emptyset$, and define $b_0 = \infty$.

Now let $n \in \omega$, and assume that the construction has been carried out for step $n$.  We now proceed with step $n+1$.  Let
\[ a_{n+1} = \inf\{y \in (a_0,\infty) \colon \exists x \in V_n \textrm{ such that } \varphi(x) = y\} .\]
Note that $a_{n+1} < b_n$ since by condition \ref{recursion 3} there exists a point $x \in V_n$ with $\varphi(x) \in [a_n,b_n)$.  Let $b_{n+1} = \min\{a_{n+1} + 2^{-(n+1)},b_n\}$.

Choose a point $x_{n+1} \in V_n$ such that $\varphi(x_{n+1}) \in [a_{n+1},b_{n+1})$.  Let $V_{n+1}$ be a clopen neighborhood of $x_{n+1}$ in $C$ of diameter less than $2^{-(n+1)}$, and (using the fact that $\varphi$ is USC) such that $\varphi[V_{n+1}] \cap [a_0,\infty) \subset [a_{n+1},b_{n+1})$.  This completes the recursive construction.

Now, the sets $V_n \cap \varphi^\leftarrow[[a_0,\infty)]$, $n \in \omega$, are nested non-empty closed subsets of $C$, hence there exists a point $z \in \bigcap_{n \in \omega} \left( V_n \cap \varphi^\leftarrow[[a_0,\infty)] \right)$.  For each $n \in \omega$, let
\[ W_n = (V_n \times [a_0,\infty)) \cap G_0^\varphi .\]
Observe that $\langle z,\varphi(z) \rangle \in W_n$, and $W_n$ is clearly closed by definition since $V_n$ is closed.  $W_n$ is open as well, since $V_n$ is also open and $(V_n \times \{a_0\}) \cap G_0^\varphi = \emptyset$.  By condition \ref{recursion 3}, $W_n \subset V_n \times [a_n,b_n)$, hence by condition \ref{recursion 4} it has diameter less than $2^{-n}$.  Thus, the sets $W_n$, $n \in \omega$, form a local basis of clopen sets at $\langle z,\varphi(z) \rangle$, which is what we wanted to find.
\end{proof}

From Theorem \ref{thm:E(X) not dense}, we obtain a new characterization of the Lelek fan: it is the unique smooth fan whose set of endpoints is homeomorphic to complete Erd\H{o}s space.  As such, we extend a result of W.\ Charatonik's (\cite[Corollary]{charatonik-lelek}):

\begin{coro}
\label{coro:Lelek fan characterization}
The following are equivalent for a smooth fan $X$ with vertex $p$:
\begin{enumerate}[label=(\roman*)]
    \item $X$ is homeomorphic to the Lelek fan;
    \item $E(X)$ is dense in $X$;
    \item $E(X) \cup \{p\}$ is connected;
    \item Every non-degenerate confluent image of $X$ is homeomorphic to it;
    \item Every non-degenerate monotone image of $X$ is homeomorphic to it; and
    \item $E(X)$ is homeomorphic to complete Erd\H{o}s space.
\end{enumerate}
\end{coro}

We remark that the analog of Theorem \ref{thm:E(X) not dense} for smooth dendroids does not hold.  For a counterexample, consider a smooth dendroid consisting of two copies of the Lelek fan whose tops are joined by an arc.  In this dendroid, the set of endpoints is not dense, but it is $1$-dimensional at every point.

\section*{Acknowledgements}
We wish to thank the referee for their valuable suggestions which helped us improve the paper.

The second named author was partially supported by NSERC grant RGPIN-2019-05998. We would also like to thank the Department of Mathematics and the Division of Basic Sciences and Engineering of the \textit{Universidad Autónoma Metropolitana, Iztapalapa} for funding the second named author's visit to Mexico city during May, 2019.

\bibliographystyle{amsplain}
\bibliography{2022-endpoint_rigid_fan.bib}

\end{document}